\documentclass[12pt,a4paper, reqno]{amsart}
\usepackage{amsfonts,amsmath,amsthm,amssymb, amscd}
\usepackage[top=2cm, bottom=5cm, left=3cm, right=1cm]{geometry}
\def\gp#1{\langle#1\rangle}

\usepackage{xcolor}
\usepackage[active]{srcltx}

\theoremstyle{plain}
\newtheorem{theorem}{Theorem}
\newtheorem{lemma}{Lemma}
\newtheorem{corollary}{Corollary}

\begin{document}

\title{Some ranks of modules over group rings}

\author{V.A.~Bovdi}
\author{L.A.~Kurdachenko}
\dedicatory{Dedicated to Professor Igor Subbotin on the occasion of his $70^{th}$ birthday}

\address{UAEU, Al Ain, United Arab Emirates}
\email{vbovdi@gmail.com}
\address{Department of Algebra and Geometry,
School of Mathematics and Mechanics,
National University of Dnipro, Ukraine}
\email{lkurdachenko@i.ua}

\begin{abstract}
A commutative ring  $R$  has {\it finite rank}  $r$, if each ideal of  $R$  is generated  at most by $r$  elements.  A commutative ring  $R$  has the {\it $r$-generator property}, if each finitely generated ideal of  $R$  can be generated by  $r$  elements. Such rings are closely related  to Pr\"ufer domains. In the present  paper we investigate  some analogs of these concepts for modules over  group rings.
\end{abstract}

\maketitle

\section{Introduction and results}

The concept of dimension of a vector space is  one of the most important concepts of mathematics. It is the source of many important numerical invariants in many areas of mathematics.

Modules are  natural generalizations  of  vector spaces, therefore, it is  natural to extend the concept of dimension to  them. One possibility is the concept of  the $R$-rank of an  $R$-module  $A$, where $R$ is a ring.
It is  defined as the cardinality of a maximal  $R$-free subset of  $A$. This  definition turned out to be effective only for the case when  $A$ is  an $R$-torsion free module over an integral domain $R$.   Other extensions of the concept of dimension are  related to the number of generators.
Unlike in the case of vector spaces, not each  submodule of a finitely generated module is finitely generated. Moreover,   when a submodule is f.g. (finitely generated),  it can have  a minimal generating set whose  cardinality  is greater than the number of elements in the generating system of the entire module. The reasons of this phenomena  are related to  the structure of the ring $R$, over which the module is considered.
Not evry left (or right) ideal of a ring is finitely generated.  Rings having such  property are called Noetherian. However, even in notherian rings the number of generators of left ideals  are not always bounded. Following  Cohen \cite{CI1950},  a commutative ring  $R$  has {\it finite rank}  $r$, if each ideal of  $R$  has $r$  generators. Almost at the same time, the concept of  special rank arises in  group theory (Maltsev \cite{MAL1948}). Dedekind and principal ideal domains are examples of such rings. Commutative rings, having finite rank, were studied in  \cite{CI1950, GR1969, GR1972, GR1973}. The notion of  rings of finite rank was generalized by Gilmer \cite{GR1969}. Following Gilmer,  we say that a commutative ring  $R$  has an {\it $r$-generator property}, if each  finitely generated ideal of  $R$  can be generated by  $r$  elements. Examples of such rings are  the Bezout domains  with  the $1$-generator property.  Commutative rings having the $r$-generator property have been actively  studied in  \cite{GR1969, GR1972, GR1973}.  In particular, such rings proved to be closely related  to Pr\"ufer domains.In our investigation  we consider analogs of these concepts for modules over  group rings. Since the notion of  "$R$-rank" is already used in  representation theory, we  adapt  the terminology from  group theory as it  was  proposed   by Maltsev in \cite{MAL1948}.  A module  $A$ over a ring  $R$  has  {\it special  rank}  $r$  if each f.g. (finitely generated)  $R$-submodule of  $A$ can be generated by  $r$  elements and there exists at least one  f.g.  $R$-submodule  $D$  of  $A$,  which has a minimal  generating subset, consisting exactly of   $r$  elements.  The special  $R$-rank and the  classical  $R$-rank  of  $A$  are denoted  by  $\mathrm{sr}_R(A)$ and $\mathrm{rank}_R(A)$, respectively.  If  $\mathrm{sr}_R(A)$ is finite, then we write   $\mathrm{sr}_{R}(A)\in \mathbb{N}$, where $\mathbb{N}$ is the set of positive integers.

Note that in the theory of groups, the topic  related to groups of finite special $FG$-rank have played and continue to play a significant role. A huge array of articles has been devoted to this subject, in which many interesting and profound results have been obtained. An excellent  overview of the basic results obtained in this topic is contained in the  article of  Dixon \cite{Dixon}. The book \cite{DKS2017} is devoted to the presentation of many  results connected  to groups of finite special rank.

The first natural  step  in  our investigation  is  the case when $RG$ is the group algebra of a finite group $G$ over a field $R$. Our main result is the following.

\begin{theorem}\label{T:1}
Let  $FG$  be the group algebra  of a  finite group $G$  over a  field $F$ of characteristic  0. A  right  $FG$-module $A$ has a finite special $R$-rank $\mathrm{sr}_{FG}(A)=r\in \mathbb{N}$  if and only if the following conditions hold:
\begin{itemize}
\item[(i)] $A = \oplus_{  1 \leq j \leq n} H_j$ with $H_j = \oplus_{  1 \leq k \leq t(j)} D_k$, where each $D_k$  is a simple  right $FG$-submodule of  $A$  and  $D_1 \cong_{FG} D_2\cong_{FG}\cdots \cong_{FG} D_{t(j)}$;
\item[(ii)]  $t(j) \leq r$ for all  $1 \leq j \leq n$  and there exist    $s\in \mathbb{N}$  such that  $t(s) = r$;
\item[(iii)]  $n$ is bounded by  the number $\mathrm{nns}_{FG}(G)$ of pairwise non-isomorphic simple  $FG$-modules.
 \end{itemize}
In particular, if $\mathrm{sr}_{FG}(A)=r\in \mathbb{N}$, then     \quad  $\dim_F(A)\leq r\cdot |G|\cdot  \mathrm{nns}_F(G)$.
\end{theorem}

If $G$  is a finite group and  $R$  is an integral domain of characteristic $0$, then a description   of  $R$-torsion-free $RG$-modules, having  finite special rank,   is given by the following.

\begin{theorem}\label{T:2}
Let  $RG$  be the group ring  of a  finite group $G$  over an integral domain $R$ of characteristic  0 and let $F$   be the  field of fractions of  $R$.
Let $A$  be a  right $RG$-module, which is  $R$-torsion free. The $RG$-module  $A$  has special rank  $\mathrm{sr}_{RG}(A)=r\in \mathbb{N}$  if and only if  the  right $FG$-module    $A \otimes_R F$  has the same  special rank  $\mathrm{sr}_{FG}(A \otimes_R F)=r$. In particular,  $\mathrm{rank}_{RG}(A) \leq \mathrm{nns}_F(G)$, where $\mathrm{nns}_{FG}(G)$ is the number of pairwise non-isomorphic simple  $FG$-modules.
\end{theorem}

Let  $A$  be a  right module over an integral domain  $R$. If  $A$  has a finite  $R$-rank, then the  section $A/\omega(RG)A$  also has a finite  $R$-rank, where $\omega(RG)=\gp{g-1\mid g\in G}_R$ is the augmentation ideal of $RG$.  Since the action of $G$ on the  right module
$A/\omega(RG)A$ is trivial, $A/\omega(RG)A$ is an $R$-module. Therefore, it is  natural to consider modules over $RG$ having finite special ranks in which  $R$ is a commutative ring with the $r$-generator property.
An important example is  the case when $R$ is a Dedekind domain. Our next result is the following.

\begin{theorem}\label{T:3}
Let  $DG$  be the group ring   of a  finite group $G$  over a  Dedekind domain $D$, all whose capitals have characteristic  $0$. Let $F$ be the field of fractions of the ring $D$.   Let $A$ be a  right  $DG$-module which is $D$-periodic.  The  module  $A$ has  special  rank  $\mathrm{sr}_{DG}(A)=r\in \mathbb{N}$  if and only if the following conditions hold:
\begin{itemize}
\item[(i)]  $A = \oplus_{P \in \pi}A_P$  where   $\pi = \mathrm{Ass}_R(A)$ and  $A_P$  is the  $P$-component of $A$;
\item[(ii)]   $A_P = \oplus_{1 \leq j \leq n(P)} H_{j, P}$  where  $n(P) \leq \mathrm{nns}_F(G)$ and each $H_{j, P}$  is a  $G$-homogeneous  $DG$-submodule of  $A$ and it  is Artinian as  a $D$-module;
\item[(iii)]  $\Omega_{[P, 1]}(H_{j, P})$  is a direct sum of at most  $r$  simple  $DG$-submodules of  $A$;
\item[(iv)]  there exist $t\in \mathbb{N}$  and a maximal ideal  $P$  such that  $\Omega_{[P, 1]}(H_{t,P})$  is a direct sum of $r$ copies of  simple  $DG$-submodules of  $A$.
\end{itemize}
\end{theorem}

All necessary definitions and notations are given in the next section.

\section{Preliminaries and Lemmas}
In the sequel, each module means a right  module, unless otherwise specified. Let  $FG$ be the group algebra of a finite group $G$ over  a field $F$ of characteristic  $0$. Each   $FG$-module  $A$ is semisimple  (see \cite[Corollary 5.15]{KOS2007}), so the algebra  $FG$  is semisimple too.  The number of pairwise non-isomorphic simple  $FG$-modules is denoted by    $\mathrm{nns}_F(G)$. When $FG$ is a finite dimensional  $F$-algebra,  $\mathrm{nns}_F(G)$ is a finite number (see \cite{BS1952, BS1956})  which  was calculated in  the following way:  Let $\xi$ be  a primitive $n$-th root of unity, where  $n$ is the greatest divisor of $\exp(G)$ which is not divisible by $char(F)$. Two  elements  $a, b \in G$  are  called  {\it $F$-conjugate} if  there exist  $x \in  G$  and   $m\in \mathbb{N}$  such that  $x^{-1}ax =b^m$ and  the map  $\xi\mapsto \xi^m$  is extensible to an automorphism of the field  $F(\xi)$  fixing  the subfield  $F$  elementwise.

Berman \cite{BS1956} and Witt \cite{Witt}  proved that  $\mathrm{nns}_F(G)$  is equal to the number of $F$-conjugate classes of  $F$-regular elements of   $G$ (see \cite[Theorem 42.8 p.\,306]{Curtis_Reiner}).   If  $F$ is   algebraically closed, then  Brauer
\cite{BR1956, Curtis_Reiner} shows that  $\mathrm{nns}_F(G)$  is equal to the number of the conjugate classes of  elements of   $G$ (see \cite[Theorem 40.1 p.\,283]{Curtis_Reiner}).  In both  cases   $\mathrm{nns}_F(G)$ is bounded by the number of conjugacy classes of $G$.

Let  $A$  be a module over an integral domain  $R$ and let $G$ be a finite group.  The set
\[
\mathrm{Ass}_R(A) = \{ P \;\mid P \quad  \text{ is a prime ideal of }  R  \text{  such that }\quad   \mathrm{Ann}_A(P)\not= \gp{0} \}
\]
denotes  the  {\it $R$-assassinator of}  $A$. The  $R$-submodule
\[
A_U  = \{ a \in  A \; \mid\quad  a U^n  = \gp{0} \quad \text{ for some }\quad    n\in\mathbb{N} \}
\]
of $A$ is called  the  {\it $U$-component}  of  $A$.  If  $A  =  A_U$, then  $A$  is   called  the {\it  $U$-module}.

For each $U$-module $A$ and each $k\in \mathbb{N}$ we define the following $R$-submodules of $A$:
\begin{equation}\label{FddSR}
\Omega_{[U, k]} (A) = \{ a\in  A  \mid  a U^k = \gp{0} \}\qquad  \text{and} \qquad
A_U  = \cup_{ s \in  \mathbb{N}}  \Omega_{[U, s]}(A).
\end{equation}
Obviously,\quad  $\Omega_{[U, 1]} (A) \leq  \Omega_{[U, 2]} (A) \leq  \cdots \leq A_U$.

Let  $A$ be a  $DG$-module such  that  $A$ is a  $P$-module for some maximal ideal  $P$  of a Dedekind domain $D$.  The module   $A$ is called  {\it  $(G, P)$-homogeneous} if  $A$ has an ascending series of  $DG$-submodules whose factors are isomorphic as  $DG$-modules.    A  {\it capital} of a Dedekind domain  $D$  is a factor-ring  $D/P$ in which   $P$  is a maximal ideal of  $D$.

Let $A$ and $B$  be  $R$-modules over an integral domain $R$. The notation   $A_R$ (or $(A)_R$) denotes the fact that $A$ is an $R$-module. Define the following $R$-submodule:
\[
\mathrm{Tor}_R(A) = \{ a \in  A \mid  \mathrm{Ann}_R(a)\not= \gp{0}  \}\leq A_R.
\]
We denoted the annihilator and the left annihilator of a module $A$  by $\mathrm{Ann}_R(A)$ and  $\mathrm{Ann}_R^{left}(A)$, respectively.    The $R$-submodule of $A$ generated by the elements $a_1,\ldots,a_n\in A$ is denoted  by $\gp{a_1,\ldots,a_n}_R$.  The $R$-isomorphism between $R$-modules $A$ and $B$ is denoted by $A\cong_RB$.

The module   $A$ is called  {\it periodic} as an $R$-module  (or simply,     $R$-periodic), if  $\mathrm{Tor}_R(A) = A$. In other words, $A$ is $R$-periodic if  $\mathrm{Ann}_R(a)\not= \gp{0}$  for each $a \in  A$. If  $\mathrm{Tor}_R(A) = \gp{0}$, then we say that  $A$ is  {\it $R$-torsion-free}. The intersection  $\Phi(A)$  of all maximal  $R$-submodules of  $A$ is called the   {\it Frattini submodule}  of  $A$. Of course, if  $A$ does not include proper maximal submodules, then   $\Phi(A) = A$.

In the sequel,  we use  the notions and results from the books \cite{Bovdi_book, Curtis_Reiner, DKS2017}.

We start our proof with the following.

\begin{lemma}\label{L:1}  Let  $A$  be an  $R$-module over a ring $R$ such that $\mathrm{sr}_{R}(A)$ is finite. Let $B$ and $C$   be   $R$-submodules of  $A$  such that  $B \subseteq  C$. The following conditions hold:\newline
\begin{itemize}
\item[(i)]\quad  $\mathrm{sr}_{R}(B)\leq \mathrm{sr}_{R}(C)\leq \mathrm{sr}_{R}(A)$;\qquad
\item[(ii)]\quad $\mathrm{sr}_{R}(C/B)\leq \mathrm{sr}_{R}(A/B)\leq \mathrm{sr}_{R}(A)$.
\end{itemize}
\end{lemma}

These assertions are obvious.

\begin{lemma}\label{L:2}
Let $B$  be an $R$-submodule  of an   $R$-module  $A$  over  a Noetherian ring $R$. If   $\mathrm{sr}_{R}(B)$ and  $\mathrm{sr}_{R}(A/B)$ are finite numbers, then \quad $\mathrm{sr}_{R}(A)\leq \mathrm{sr}_{R}(B) +\mathrm{sr}_{R}(A/B)$.
\end{lemma}

\begin{proof}
Let  $D$  be a f.g.   $R$-submodule of  $A$. Clearly,   $D/(D \cap B) \cong_R  (D + B)/B$  has a  finite generating subset  $\{ d_j+ (D \cap B)\mid  1 \leq j \leq m \}$ with  $m \leq \mathrm{sr}_{R}(A/B)$. Since $D$ is a f.g. module over a Noetherian ring, $D$ is a Noetherian $R$-module (see \cite[Lemma 1.1]{KOS2007}), so   $D \cap B$  is a f.g.  $R$-submodule of  $B$. Since  $\mathrm{sr}_{R}(B)$ is finite,   there exists a finite  set of generators   $\{ b_j\mid  1 \leq j \leq k \}$ of $D \cap B$ with   $k \leq \mathrm{sr}_{R}(B)$. The  subset   $\{ d_j \mid 1 \leq j \leq m \}  \cup \{ b_j\mid  1 \leq j \leq k \}$  generates  $D$  as an  $R$-submodule and  $m + k \leq \mathrm{sr}_{R}(B) +\mathrm{sr}_{R}(A/B)$.
\end{proof}

\begin{lemma}\label{L:3}
Let $A = \oplus_{1\leq j\leq n} A_j$  be an  $R$-module over a ring $R$   in which  each $A_j$  is a simple  $R$-submodule.  If $\mathrm{Ann}_R^{left}(A_j)\not= \mathrm{Ann}_R^{left}(A_i)$  for all  $i\not= j$, then  $A$  is a cyclic  $R$-module.
\end{lemma}
\begin{proof}
Since each  $A_j$  is a simple   $R$-submodule, $U_j=\mathrm{Ann}_R^{left}(A_j)$  is a maximal left ideal of  $R$.
From  $U_i\not= U_j$  follows that   $U_i + U_j = R$  for all  $i\not= j$.

Let us prove that   $A=\gp{a_1 + \cdots+ a_n\mid 0\not=a_i\in A_i}_R$. Indeed, $U_1 + U_2 = R$ shows that  we can choose  $\alpha_1 \in  U_2$ such that $\alpha_1 \not\in U_1$. Clearly   $\alpha_1 a_1\not= 0$,  $\alpha_1 a_2 = 0$  and
\[
\alpha_1(a_1 +a_2 + \cdots+ a_n) =  \alpha_1 a_1 + \alpha_1 a_3 + \cdots+ \alpha_1 a_n.
\]
If  $\alpha_1 a_3 + \cdots+ \alpha_1 a_n = 0$, then  $\alpha_1 a_1 \in  R(a_1 + \cdots+ a_n)$. Since  $A_1$  is a simple  $R$-submodule, it is generated by any  non-zero element, so  $A_1 = Ra_1 \leq R(a_1 + \cdots+ a_n)$.

Assume now that  $\alpha_1 a_3 + \cdots+ \alpha_1 a_n\not= 0$. Let  $j_1$  be a smallest  positive integer (say, $j_1 = 3$), such that  $j_1 > 2$  and  $\alpha_1 a_{j_1}\not= 0$.  The  equality  $U_1 + U_3 = R$ shows that  we can choose $\alpha_2 \in  U_3$ with $\alpha_2 \not \in  U_1$. Hence  $\alpha_2\alpha_1 a_1\not= 0$,  $\alpha_2\alpha_1 a_3 = 0$  and
\[
\alpha_2(\alpha_1 a_1 + \alpha_1 a_3 + \cdots+ \alpha_1 a_n) = \alpha_2\alpha_1 a_1 + \alpha_2\alpha_1 a_4 + \cdots+ \alpha_2\alpha_1 a_n.
\]
If  $\alpha_2\alpha_1 a_4 + \cdots+ \alpha_2\alpha_1 a_n = 0$ then,  as above, we have   $A_1=Ra_1 \leq R(a_1 + \cdots+ a_n)$. Otherwise,   we can repeat finitely many times the mentioned  argument above. Finally  we obtain that  $A_1 \leq R(a_1 + \cdots+ a_n)$. Similarly, using the  same  argument,  $A_j=Ra_j \leq R(a_1 + \cdots+ a_n)$  for each  $1 \leq j \leq n$, so $A = A_1 \oplus  \cdots\oplus  A_n = R(a_1 + \cdots+ a_n)$.
 \end{proof}

\begin{lemma}\label{L:4}
Let  $A = \oplus_{1\leq j\leq n} A_j$   be an  $R$-module over  a ring $R$, in which   each $A_j$  is a simple  $R$-submodule.  If
\[
\mathrm{Ann}_R^{left}(A_1) =\mathrm{Ann}_R^{left}(A_2) =\cdots= \mathrm{Ann}_R^{left}(A_n),
\]
then $A_j\cong_R Ra$ for  each  $0\not= a\in  A$  and    $1 \leq j \leq n$.
\end{lemma}

\begin{proof} Let  $a = a_1 +\cdots + a_n\in A$  where  each $a_j \in  A_j$  for $ 1 \leq  j \leq  n$. Without loss of generality,  we can assume   $a_1 \not= 0$. Put  $B= Ra$  and  $D= A_2 \oplus  \cdots\oplus  A_n$. Then  $(B + D)/D$  is a non-zero submodule of  $A/D$.  The  isomorphism  $A/D \cong_R A_1$  shows that  $A/D$  is a simple  $R$- module. It follows that  either
$A/D = (B + D)/D$  or  $A = B + D$. Using   the  isomorphism  $(B + D)/D \cong_R B/(B \cap D)$,   we obtain that  $B/(B \cap  D) \cong_R A_1$. Suppose that  $B \cap D\not=0$  and choose  $0\not=c \in B\cap D$. Since  $c \in  B = Ra$ and   $c = ya$  for some   $y \in  R$,
$c =ya_1 + \cdots  + ya_n$. On the other hand, $c \in D$, so that $pr_1(c) = ya_1 = 0$ in which  $pr_j(A)=A_j$. It follows that   $y \in  \mathrm{Ann}_R^{left}(a_1)$.  Since  $A_1$  is a simple   $R$-submodule and  $a_1 \not= 0$,   $\mathrm{Ann}_R^{left}(a_1) = \mathrm{Ann}_R^{left}(A_1)$. This yields   $y \in \mathrm{Ann}_R^{left}(A_j)$  for each  $1 \leq  j \leq  n$. Consequently,   $ya_j  = 0$  for  all  $j$, so   $c = 0$ is  a contradiction. It  proves that  $B \cap  D = \gp{0}$  and  $B\cong_R A_1$.
\end{proof}

\begin{corollary}\label{C:1}
Let  $A = \oplus_{\lambda \in \Lambda} A_\lambda$  be an  $R$-module over a ring $R$ in which each  $A_\lambda$  is a simple  $R$-submodule, where $\Lambda$ is a linearly  ordered set.  If  $\mathrm{Ann}_R^{left}(A_\lambda) = \mathrm{Ann}_R^{left}(A_\mu)$  for all  $\lambda, \mu \in \Lambda$, then   $Ra\cong_RA_\delta$, for each  $0\not= a \in  A$ and $\delta\in \Lambda$.
\end{corollary}

\begin{corollary}\label{C:2}
Let  $A =\oplus _{1 \leq j \leq n} A_j$  be an  $R$-module over a ring $R$ in which   each $A_j$  is a simple  $R$-submodule. If $A_1 \cong_R A_2\cong_R \cdots \cong_R A_n$, then $\mathrm{sr}_{R}(A)=r\in \mathbb{N}$   if and only if  $n = r$.
\end{corollary}

\begin{proof} Let  $B=\gp{S}_R$  be a  f.g.  $R$-submodule of  $A$ with a minimal generating subset  $S = \{ b_1, \ldots, b_k \}$.  The $R$-submodule   $B_1=\gp{b_1}_R \cong_R A_1$  by Lemma \ref{L:4}. Since  $B$  is a  semisimple  $R$-module (see \cite[Corollary 4.4]{KOS2007}), there exists  an  $R$-submodule  $D_1$  such that  $B = B_1 \oplus  D_1$. Clearly $b_2\in D_1$ and $b_2\not \in B_1$, by the minimality of $S$.
The $R$-submodule $B_2=\gp{b_2}_R$ is isomorphic to $A_1$  by  Lemma \ref{L:4}. Repeating this argument finitely many times, we obtain   a  direct decomposition  $B = B_1 \oplus  B_2 \oplus  \cdots\oplus  B_k$,   where  each $B_j \cong_R A_1$.

Choosing a different minimal generating subset  $\{ d_1, \ldots, d_m \}$ of   $B$ and   repeating the same  argument, we construct  a direct decomposition  $B = D_1 \oplus  \cdots\oplus  D_m$  in which each $D_j\cong_R A_1$.  Clearly $k = m$  by the Krull-Remak-Schmidt theorem (see  \cite[Chapter 6, Theorem 1.6]{CP1991}). In particular, if  $B = A$, then   $n \leq r$ and  the above argument shows that  $n = r$. \end{proof}

\bigskip

Let  $A$  be an  $R$-module over  a ring $R$, such that $A =\oplus_{1 \leq j \leq n} A_j$  in which   each $A_j$  is a simple  $R$-submodule. The relation  "$A_i \cong_R A_j$"  is an equivalence relation on the set  $\{ A_j \mid 1 \leq j \leq n \}$ with   equivalence classes $E_1, \ldots, E_k$. The direct sum of all submodules from the  set $E_j$   is  denoted  by  $S_j$ and it is called  {\it the homogeneous component} of  $A$.

\begin{corollary}\label{C:3}
Let  $A =\oplus_{1 \leq j \leq n}A_j$  be an  $R$-module over  a ring $R$, in which   each $A_j$  is a simple  $R$-submodule and let $E_1,\ldots, E_k$ be the homogeneous components of  $A$. The number  $r=\mathrm{sr}_{R}(A)$ is finite    if and only if the following conditions hold:
\begin{itemize}
\item[(i)] $\mathrm{sr}_{R}(E_j)\leq r$ for each   $j\in\{1,\ldots,k\}$;
\item[(ii)]  there exist at least one $j\in\{1,\ldots,k\}$ such that  $\mathrm{sr}_{R}(E_j)=r$.
\end{itemize}
\end{corollary}

\begin{proof} Clearly    $A = E_1 \oplus \cdots\oplus  E_k$,  where each   $E_j = A_{j,1} \oplus  \cdots\oplus  A_{j,t(j)}$  has a special rank  $\mathrm{sr}_{R}(E_j)=t(j)\leq r$ by  Corollary \ref{C:2}.  Set  $s=\max \{ t(1), \ldots, t(k) \}$. It is easy to see  that  each  $B_i = \oplus_{1 \leq m \leq k}  A_{m,i}$ is cyclic by Lemma \ref{L:3}. It follows that  $A$ has a minimal generating subset, consisting exactly of $s$  elements. Moreover,   any  $R$-submodule  $D\leq A$  is also semisimple, so   $D = \oplus_{1 \leq l \leq q}  D_l$ in which   $q \leq n$. Furthermore,   each  simple  $R$-submodule $D_l$  is isomorphic to a submodule from the set  $\{ A_1, \ldots, A_n \}$ (see \cite[Corollary 4.4]{KOS2007}). Repeating the above arguments for  $D$, we obtain   that the  minimal generating subset of $D$  consists at most of  $s$  elements, so  $s=r$.
\end{proof}

\section{Proofs}

\begin{proof}[Proof of Theorem \ref{T:1}]
The fact $char(F) = 0$ implies that  $A$  is a semisimple  $FG$-module (see \cite[Corollary 5.15]{KOS2007}),
so  $A = \oplus_{1 \leq j \leq n} H_j$  in which   each $H_j$  is a homogeneous component of  $A$ and  $H_j$  is a direct sum of at most  $r$  simple submodules by  Corollary \ref{C:2}. As we noted above,
$n \leq \mathrm{nns}_F(G)$. Obviously, the dimension of a simple  $FG$-submodule of  $A$ is bounded by $|G|$, so   $dim_F(H_j) \leq r|G|$  and $dim_F(A) \leq r\cdot |G|\cdot \mathrm{nns}_F(G)$. The sufficiency of the conditions  of our  Theorem follows from Corollary \ref{C:3}.
\end{proof}

\begin{proof}[Proof of Theorem \ref{T:2}]
Put  $B = A \otimes_R F$. Let $C=\gp{ b_1, \ldots, b_n}_{FG}$  be an   $FG$-submodule of $B$ generated by the subset  $\{ b_1, \ldots,  b_n\}\subseteq A$. For each  $b_j$  there exists a non-zero   $y_j\in  R$  such that  $a_j = y_jb_j \in  A$.
This choice ensures that the $FG$-submodule $\gp{a_1, \ldots, a_n}_{FG}$ of  $B$ coincides with  $C$. Let  $E=\gp{ a_1, \ldots, a_n }_{RG}$  be an  $RG$-submodule.  Clearly,  $C/E$  is  an  $R$-periodic $R$-module. The fact that  $A$ has a finite special rank  $r$  implies that  $E$  has an  $RG$-generating subset  ${ e_1, \ldots, e_r}$ in which   $r \leq n$. The construction of  $E$  implies that  $C=\gp{ e_1, \ldots, e_r}_{FG}$. Thus each finitely generated  $FG$-submodule of  $B$  can be generated by  at most  $r$  elements. It follows that the  $FG$-module  $B$  has a finite special rank at most  $r$. Moreover, using  similar arguments and Theorem 1, we can prove  that  $B$  has a special rank  $r$ and   $B$  has a  finite dimension  $d$  at most  $r|G|\mathrm{nns}_F(G)$. Let  $\{v_1, \ldots, v_d \}$  be a basis of  B. Again, for each $v_j$  we can choose a non-zero  $z_j \in  R$  such that  $u_j =z_jv_j \in  A$. Since  $\{ v_1, \ldots, v_d \}$  is a basis  of  $B$,  $\{ u_1, \ldots, u_d \}$  is a maximal  independent subset of  $A$. Consequently,  the  $R$-rank of  $A$  is exactly  $d$.
\end{proof}

\begin{lemma}\label{L:5}
Let $DG$ be the group ring    of a finite group $G$ over  a Dedekind domain $D$ of characteristic  $0$. Let $A$  be a  $DG$-module which is also a $P$-module for some maximal ideal  $P$  of  $D$ with the property that  $char(D/P) = 0$. If  $B$  is  a  $DG$-submodule of  $A$  such that  $A = B \oplus  C$  for a $D$-submodule  $C$ of $A$,  then there exists a  $DG$-submodule  $K$  with the property  that  $A = B \oplus  K$.
\end{lemma}

\begin{proof}
Put $D_P=D/P$. Consider  the  $D$-submodule   $\Omega_{[P, 1]} (A)$ (see \eqref{FddSR}).  It is easy to see that   $\mathrm{Ann}_D(a) = P$ for each   $a \in \Omega_{[P, 1]} (A)$, so    $\Omega_{[P, 1]} (A)$ can be considered  as a  $(D/P)G$-module. In particular, the additive group of  $\Omega_{[P, 1]} (A)$  is a direct sum of isomorphic copies of the additive group   $D/P$. Since  $char(D/P) = 0$, the  additive group of  $D/P$  is divisible, so the additive group of   $\Omega_{[P, 1]} (A)$  is divisible too. Also the  additive group of   $\Omega_{[P, j + 1]} (A)/ \Omega_{[P, j]}(A)$   is divisible for each  $j$ by the same reason, so the   additive group  of  $A=\cup_{j \in \mathbb{N}}   \Omega_{[P, j]} (A)$ has an ascending series  whose factors are divisible. It follows that the additive group  of  $A$ is divisible. There exists a  $DG$-submodule  $K$  such that  $nA \leq B + K$  and  $n(B \cap K) = \gp{0}$ by  \cite[Theorem 5.9]{KOS2007}. The fact that the additive group  of  $A$  is divisible implies that  $nA = A$, so
$A =     B + K$. The fact that  $char(D/P) = 0$  implies that an additive group of  $A$ is torsion-free, so that
$B \cap K = \gp{0}$ and   $A = B \oplus  K$.
\end{proof}

\begin{lemma}\label{L:6}
Let $DG$ be the group ring of a finite group $G$ over  a Dedekind domain $D$ of characteristic  0. Let $A$  be a  $DG$-module  such that $A$ is  also   a  $P$-module for some maximal ideal  $P$  of  $D$. If $char(D/P) = 0$ and $\mathrm{sr}_{DG}(A)=r\in \mathbb{N}$,  then $A$ is an Artinian $D$-module.
\end{lemma}
\begin{proof} Put $D_P=D/P$. The  $D$-submodule  $\Omega_{[P, 1]} (A)$  is a  $DG$-submodule of  $A$ and  $\mathrm{Ann}_D(a) = P$ for   each $a \in  \Omega_{[P, 1]} (A)$. Obviously,    $\Omega_{[P, 1]} (A)$  is a  $D_P[G]$-module and    $dim_{ D_P}(\Omega_{[P, 1]} (A))$ is finite by Lemma \ref{L:1} and Theorem \ref{T:1}. Thus   $A$ is Artinian as a  $D$-module   (see \cite[Corollary 7.12]{KSS2008}). \end{proof}

Let $A$ be a simple  $R$-module  over a Dedekind domain $R$ (in particular, principal ideal domain). Clearly,   $A \cong  R/P$  for some maximal ideal   $P$. The factors  $R/P^k$  and  $P/P^{k + 1}$  are isomorphic as  $R$-modules for each   $k\in\mathbb{N}$ (see  \cite[Corollary 1.28]{KSS2008}). In particular, the  $R$-module  $R/P^k$  is embedded in the $R$-module  $R/P^{k+1}$ for each   $k \in  \mathbb{N}$ and $\{ R/P^k \mid  k \in  \mathbb{N} \}$ is an injective family of  $R$-modules. Consider the $R$-module $C_{P^\infty} =  \lim_{inj}  \{ R/P^k \mid   k \in  \mathbb{N}\}$ which is called the {\it Pr\"ufer  $P$-module}. It is easy to see that  $C_{P^\infty}$   is a  $P$-module, $\Omega_{[P, k]}(C_{P^\infty}) \cong_RR/P^k$ and
\[
\Omega_{[P, k + 1]} (C_{P^\infty})/\Omega_{[P, k]}(C_{P^\infty}) \cong_R  (R/P^{k + 1})/(P/P^{k + 1}) \cong_R  R/P.
\]
Hence, if  $C$  is a proper  $R$-submodule of  $C_{P^\infty}$ (i.e. $C \not= C_{P^\infty}$),   then  $C =\Omega_{[ P, k]} (C_{P^\infty})$ for some $k\in \mathbb{N}$. Similarly, if  $b \not\in \Omega_{[P, k-1]}  (C_{P^\infty})$,  then $C = Rb$.
Observe also that a  Pr\"ufer  $P$-module is monolithic and its monolith coincides with  $\Omega_{[P, 1]} (C_{P^\infty})$. Indeed, the intersection $C\cap \Omega_{[P,1]}(C_{P^\infty})$ is non-zero for each $R$-submodule $C$ of a Pr\"ufer $P$-module. Since $\Omega_{[P,1]}(C_{P^\infty})$ is a simple $R$-module ($\Omega_{[P,1]}(C_{P^\infty})$ is isomorphic to $R/P$),
\[
C\cap \Omega_{[P,1]}(C_{P^\infty})=\Omega_{[P,1]}(C_{P^\infty}).
\]
It follows that the $R$-monolith of a Pr\"ufer $P$-module includes $\Omega_{[P,1]}(C_{P^\infty})$ and, hence coincides with it, because $\Omega_{[P,1]}(C_{P^\infty})$ is a simple $R$-submodule.

Indeed, let $C$ be a finitely generated $R$-submodule of a Pr\"ufer $P$-module. Then C is a proper $R$-submodule. As we have seen above,  $C$ coincides with $\Omega_{[P,n]}(C_{P^\infty})$ for some $n\in \mathbb{N}$. The last $R$-submodule is cyclic, it is isomorphic  to $R/P^n$. Hence each f.g. $R$-submodule of a Pr\"ufer $P$-module is cyclic. It follows that a Pr\"ufer $P$-module has $R$-rank 1.

\begin{corollary}\label{C:4}
Let $DG$ be the group ring    of a finite group $G$  over  a Dedekind domain $D$. Let $A$  be a  $DG$-module of $\mathrm{sr}_{DG}(A)=r\in \mathbb{N}$ such that it is also   a  $P$-module for some maximal ideal  $P$  of  $D$. If $char(D/P) = 0$, then $A = K \oplus  B$ in which   $K, B$  are  $DG$-submodules. Moreover,  $K$  is a f.g.   $D$-submodule and  $B$  is a direct sum of finitely many Pr\"ufer  $P$-submodules.
\end{corollary}

\begin{proof} The module $A$ is Artinian as a $D$-module by Lemma \ref{L:6}, so $A = B \oplus  C$, where $B$  is a direct sum of finitely many Pr\"ufer  $P$-submodules and  $C$  is a f.g.  $D$-submodule. Clearly  $B$  is a  $DG$-submodule, so  there exists a  $DG$-submodule  $K$  with  $A = B \oplus  K$ by Lemma \ref{L:5}.

Since  $K \cong_D C$, the module  $K$  is finitely generated as a  $D$-submodule. \end{proof}

\begin{lemma}\label{L:7}
Let $DG$ be the group ring    of a finite group $G$  over  a Dedekind domain $D$. Let  $B$  be a  $DG$-submodule of a  $DG$-module  $A$  which satisfies   the following conditions:
\begin{itemize}
\item[(i)]  $B = \oplus_{ \lambda \in  \Lambda} B_\lambda$  where  $B_\lambda$  is a simple  $DG$-submodule for all  $\lambda \in  \Lambda$;
\item[(ii)]  $A/B$  is a simple  $DG$-module and  $A/B\not\cong_{DG} B_\lambda$ for all   $\lambda\in \Lambda$;
\item[(iii)]  $\mathrm{Ann}_D(B) = \mathrm{Ann}_D(A/B) = P$  is a maximal ideal of  $D$;
\item[(iv)]  $char(D/P) = 0$.
\end{itemize}

Then there exists a  $DG$-submodule  $K$  such that  $A = B \oplus  K$.
\end{lemma}

\begin{proof}
Let us fix  $y \in  P \setminus P^2$. Consider  a map  $f: A\to A$ such that  $f(a) = ya$ for  $a \in  A$. The map  $f$  is a   $DG$-endomorphism and $yA = Im(f) \leq B$. Indeed,  $Ann_D(A/B) = P$ by (iii),  $y\in P \setminus P^2\subset P$ and  $y\in Ann_D(A/B)$, so $yA\leq B$.

If   $yA\not= \gp{0}$,  then  $Ker(f) = B$ by condition (ii) and $yA = Im(f) \cong_{ DG} A/Ker(f) = A/B$,  so that  $yA$  is a simple  $DG$-submodule of  $B$. Thus  $yA \cong_{DG} B_\lambda$  for some  $\lambda \in  \Lambda$ (see \cite[Corollary 4.4]{KOS2007}), which contradicts condition (ii).
Consequently,   $yA = \gp{0}$ and   $PA = \gp{0}$. It follows that  $A$ is  a vector space over  $D/P$ and   $A=B\oplus  C$  for a   $D$-submodule $C$ by Lemma \ref{L:5}.
\end{proof}

\begin{corollary}\label{C:5}
Let  $DG$ be a group ring of a finite group $G$ over  a Dedekind domain $D$. Let  $A$  be a  $DG$-module such that  the following conditions hold:
\begin{itemize}
\item[(i)]  $\mathrm{sr}_{DG}(A)=r\in \mathbb{N}$;\
\item[(ii)] $A$  is    a  $P$-module for some maximal ideal  $P$  of  $D$;
\item[(iii)] $A$ is a f.d.  $D$-module.
\end{itemize}
If  $char(D/P) = 0$, then  $A = \oplus_{1 \leq j \leq n} L_j$ in which    each $L_j$  is a  $(G, P)$-homogeneous  $DG$-submodule of  $A$.
\end{corollary}

\begin{proof}  As before, the  $(D/P)G$-module $\Omega_{[P, 1]} (A)$    is a semisimple  $DG$-submodule by  Lemma \ref{L:1} and Theorem \ref{T:1}. Consequently,   $\Omega_{[P, 1]} (A)  =  \oplus_{ 1 \leq j \leq n} H_j$ in which each  $H_j$  is a homogeneous component of  $\Omega_{[P, 1]} (A)$ and  $A =\oplus_{  1 \leq j \leq n} L_j$ in which  each $L_j$  is a  $(G, P)$-homogeneous  $DG$-submodule of  $A$ such that  $\Omega_{[P, 1]}(L_j) = H_j$.
\end{proof}

\begin{corollary}\label{C:6}
Let  $DG$ be the  group ring of a finite group $G$ over  a Dedekind domain $D$.
Let  $A$  be a  $DG$-module such that  the following conditions hold:
\begin{itemize}
\item[(i)]  $\mathrm{sr}_{DG}(A)=r\in \mathbb{N}$;\
\item[(ii)] $A$  is    a  $P$-module for some maximal ideal  $P$  of  $D$;
\end{itemize}
If  $char(D/P) = 0$, then  $A = \oplus_{ 1 \leq j \leq n} L_j$  in which  each  $L_j$  is a  $(G, P)$-homogeneous  $DG$-submodule of  $A$.
\end{corollary}

\begin{proof} Since  $A$ is a  $P$-module, $A  = \cup_{ j \in \mathbb{N}}  \Omega_{[P, j]} (A)$. Using induction and  Corollary \ref{C:5} we get the result.
\end{proof}

Let     $A$ be a  $D$-module over  a Dedekind domain $D$. The intersection $\Phi(A)$  of all maximal  $D$ - submodules of  $A$ is called the {\it Frattini submodule}  of  $A$. Of course, if  $A$  does not include proper maximal submodules, then $\Phi(A) = A$.

If  $A$ is a f.g. periodic module  and  $\{ a_j + \Phi(A)\mid 1 \leq j \leq n \}$  is a generating set for  $A/\Phi(A)$, then $\{ a_j\mid 1 \leq j \leq n \}$  is a generating set for  $A$.  Moreover, if  $\{ a_j + \Phi(A) \mid 1 \leq j \leq n \}$  is a minimal generating set for  $A/\Phi(A)$, then $\{ a_j\mid 1 \leq j \leq n \}$  is a minimal generating set for  $A$.

\begin{lemma}\label{L:8}
Let  $DG$ be the group ring of a finite group $G$ over  a Dedekind domain $D$. Let  $A$  be a  $DG$-module  such that  the following conditions hold:
\begin{itemize}
\item[(i)]   $A$ is  a  $P$-module for some maximal ideal  $P$  of  $D$;
\item[(ii)]  $A$ is finitely generated as  a $D$-module;
\item[(iii)] $A$ is  $(G, P)$-homogeneous;
\item[(iv)]   char(D/P) = 0.
\end{itemize}

The number  $\mathrm{sr}_{DG}(A)$ is equal to $r\in \mathbb{N}$   if and only if  $\Omega_{[P, 1]}(A)$  is a direct sum of  $r$ copies of   simple  $DG$-submodules.
\end{lemma}

\begin{proof} If  $\mathrm{sr}_{DG}(A)=r\in \mathbb{N}$, then   $\Omega_{[P, 1]} (A) = M_1 \oplus  \cdots\oplus  M_k$,  in which  each  $M_j$ is a  simple  $DG$-submodule and  $M_j \cong_{DG} M_m$  for all  $j, m$, and $ k \leq r$ by  Lemma \ref{L:1} and Theorem \ref{T:1}. Note that  $\Omega_{[P, 1]} (A) \cong_{DG} A/\Phi(A)$, so  $k = r$.  Conversely,  if  $k = r$,   then  $\mathrm{sr}_{DG}(A)=r\in \mathbb{N}$.\end{proof}

\begin{corollary}\label{C:7}   Let  $DG$ be the group ring of a finite group $G$ over  a Dedekind domain $D$. Let  $A$  be a  $DG$-module such that the following conditions hold:
\begin{itemize}
\item[(i)]   $A$ is  a  $P$-module for some maximal ideal  $P$  of  $D$;
\item[(ii)]  $A$ is finitely generated as a $D$-module;
\item[(iii)] $char(D/P) = 0$.
\end{itemize}
Then $\mathrm{sr}_{DG}(A)=r\in \mathbb{N}$  if and only if  the following conditions hold:
\begin{itemize}
\item[(a)]  $A = \oplus_{1 \leq j \leq n} H_j$ in which each   $H_j$  is a  $(G, P)$-homogeneous  $DG$-submodule of  $A$;
\item[(b)]   $\Omega_{[P, 1]}(H_j)$  is a direct sum of at most  $r$  simple  $DG$-submodules of  $A$;
\item[(c)]   there exists $t\in \mathbb{N}$  such that  $\Omega_{[P, 1]}(H_t)$  is a direct sum of $r$  simple  $DG$-submodules of  $A$.
\end{itemize}
\end{corollary}
\begin{proof}
For the  proof it is sufficient to use the arguments from the proof of  Theorem \ref{T:1} and the isomorphism  $\Omega_{[P, 1]}(A) \cong_{DG} A/\Phi(A)$.
\end{proof}

\begin{corollary}\label{C:8}
Let  $DG$ be the group ring of a finite group $G$ over  a Dedekind domain $D$. Let  $A$  be a  $DG$-module such that the following conditions hold:
\begin{itemize}
\item[(i)] $A$ is  a  $P$-module for some maximal ideal  $P$  of  $D$;
\item[(ii)]   $char(D/P) = 0$.
\end{itemize}
The number $\mathrm{sr}_{DG}(A)$ is finite and equals to $r\in \mathbb{N}$  if and only if  the following conditions hold:
\begin{itemize}
\item[(a)]  $A = \oplus_{1 \leq j \leq n} H_j$ in which  each  $H_j$  is a  $(G, P)$-homogeneous  $DG$-submodule of  $A$;
\item[(b)]   $\Omega_{[P, 1]}(H_j)$  is a direct sum of at most  $r$  simple  $DG$-submodules of  $A$;
\item[(c)]   there exists $t\in \mathbb{N}$   such that  $\Omega_{[P, 1]}(H_t)$  is a direct sum of $r$  simple  $DG$-submodules of  $A$.
\end{itemize}
In particular, if  $A$ has finite special rank, then  $A$ is Artinian as a $D$-module.
\end{corollary}

\begin{corollary}\label{C:9}
 Let  $DG$ be the group ring of a finite group $G$ over  a Dedekind domain $D$. Let  $A$  be a  $DG$-module which  satisfies the following conditions:
\begin{itemize}
\item[(i)]  $A$ is a $D$-periodic module;
\item[(ii)] $A$  is finitely generated as a  $D$-module;
\item[(iii)]  $char(D/P) = 0$  for all maximal ideal  $P\in \mathrm{Ass}_R(A)$.
\end{itemize}
The number  $\mathrm{sr}_{DG}(A)=r\in \mathbb{N}$  if and only if  the following conditions hold:
\begin{itemize}
\item[(a)]  $A = \oplus_{1 \leq j \leq n; P \in \pi}   H_{j, P}$ in which    each $H_{j, P}$  is a  $(G,P)$-homogeneous  $DG$-submodule of  $A$ and  $P\in \pi = \mathrm{Ass}_R(A)$;
\item[(b)]   $\Omega_{[P, 1]}(H_{j, P})$ is a direct sum of at most  $r$  simple  $DG$-submodules of  $A$;
\item[(c)]  there exists $t\in \mathbb{N}$  and a maximal ideal  $P$  such that  $\Omega_{[P, 1]}(H_{t, P})$  is a direct sum of $r$  simple  $DG$-submodules of  $A$.
\end{itemize}
\end{corollary}

\begin{proof} Since  $A$ is  $D$-periodic, $A = \oplus_{P \in \pi}  A_P$ in which $\pi= \mathrm{Ass}_R(A)$  and  $A_P$  is a $P$-component of  $A$. Since  $A$ is a f.g. module, the set $\pi$ is finite (see \cite[Theorem 7.8]{KSS2008}). Now we can apply Corollary \ref{C:7}.
\end{proof}

\begin{proof}[Proof of the Theorem \ref{T:3}] Since  $A$ is  $D$-periodic, $A = \oplus_{P \in \pi} A_P$  where    $\pi= \mathrm{Ass}_R(A)$ and   $A_P$  is a  $P$-component of  $A$ (see  \cite[Corollary 3.8]{KSS2008}). Now using Corollary \ref{C:8} we obtain that  $A_P$  satisfies (ii), (iii) and  (iv). Conversely, suppose that  $A$ satisfies conditions (i)- (iv). If  $B$  is a f.g.   $DG$-submodule of  $A$, then  $B$  has special rank at most  $r$ by Corollary \ref{C:8}. Consequently,   $A$ has  special rank  $r$  by (iv) and Corollary \ref{C:2}.
\end{proof}
\section{Funding}
This research was supported by the UAEU UPAR Grant G00002160.

\end{document}